\documentclass[letterpaper, 10 pt, conference]{ieeeconf}

\IEEEoverridecommandlockouts

\makeatletter
\newcommand{\removelatexerror}{\let\@latex@error\@gobble}
\makeatother
\usepackage[T1]{fontenc}
\usepackage[utf8]{inputenc}
\usepackage[english]{babel}
\usepackage{amsmath,amsthm,amssymb}
\usepackage[]{amsfonts}
\usepackage{graphicx}
\usepackage{url}
\usepackage{makecell}

\usepackage{pgfplots}
  \pgfplotsset{compat=newest}
  \usetikzlibrary{plotmarks}
  \usetikzlibrary{arrows.meta}
  \usepgfplotslibrary{patchplots}
  \usepackage{grffile}
  \usepackage{amsmath}

  \pgfplotsset{plot coordinates/math parser=false}

\usepackage{bbm}
\usepackage{bbold}

\newtheorem{theorem}{Theorem}[section]
\newtheorem{proposition}[theorem]{Proposition}
\newtheorem{lemma}[theorem]{Lemma}
\newtheorem{corollary}[theorem]{Corollary}

\theoremstyle{definition}

\newtheorem{remark}[theorem]{Remark}

\newcommand{\EE}{\mathcal{E}}
\newcommand{\GG}{\mathcal{G}}
\newcommand{\HH}{\mathcal{H}}
\newcommand{\PP}{\mathcal{P}}
\newcommand{\VV}{\mathcal{V}}
\newcommand{\WW}{\mathcal{W}}
\newcommand{\FF}{\mathcal{F}}

\newcommand{\NN}{\mathcal{N}}
\newcommand{\II}{\mathcal{I}}

\renewcommand{\SS}{\mathcal{S}}

\newcommand{\CVaR}{\operatorname{CVaR}}

\newcommand{\VI}{\operatorname{VI}}
\newcommand{\SOL}{\operatorname{SOL}}
\newcommand{\argmin}{\operatorname{argmin}}

\newcommand{\real}{\mathbb{R}}
\newcommand{\realnonnegative}{{\mathbb{R}}_{\ge 0}}
\newcommand{\naturalnumbers}{\mathbb{N}}
\newcommand{\setdef}[2]{\{#1 \; | \; #2\}}
\newcommand{\setdefb}[2]{\Bigl\{#1 \; \big| \; #2\Bigr\}}
\newcommand{\map}[3]{#1:#2 \rightarrow #3}

\newcommand{\Pb}{\mathbb{P}}
\newcommand{\norm}[1]{\ensuremath{\| #1 \|}}

\newcommand{\until}[1]{[#1]}
\newcommand{\CVaRhat}{\widehat{\CVaR}}
\newcommand{\Fhat}{\widehat{F}}
\newcommand{\cl}{\mathrm{cl}}

\usepackage[normalem]{ulem}

\newcommand{\oprocendsymbol}{\hbox{$\bullet$}}
\newcommand{\oprocend}{\relax\ifmmode\else\unskip\hfill\fi\oprocendsymbol}
\newcommand{\longthmtitle}[1]{\mbox{}\textup{\textsl{(#1):}}}

\title{Stochastic approximation for CVaR-based variational inequalities} 
\author{Jasper Verbree \qquad Ashish Cherukuri \thanks{The authors are with the Engineering and Technology Institute Groningen, University of Groningen. Email: \texttt{\{j.verbree, a.k.cherukuri\}@rug.nl.}}}

\begin{document}

\maketitle

\begin{abstract}
In this paper we study variational inequalities (VI) defined by the conditional value-at-risk (CVaR) of uncertain functions. We introduce stochastic approximation schemes that employ an empirical estimate of the CVaR at each iteration to solve these VIs. We investigate convergence of these algorithms under various assumptions on the monotonicity of the VI and accuracy of the CVaR estimate. Our first algorithm is shown to converge to the exact solution of the VI when the estimation error of the CVaR becomes progressively smaller along any execution of the algorithm. When the estimation error is nonvanishing, we provide two algorithms that provably converge to a neighborhood of the solution of the VI. For these schemes, under strong monotonicity, we provide an explicit relationship between sample size, estimation error, and the size of the neighborhood to which convergence is achieved. A simulation example illustrates our theoretical~findings.
\end{abstract}

\section{Introduction} \label{section introduciton}
Variational inequality (VI) problems find application in a broad range of areas~\cite{FF-JSP:03}, e.g., in game theory, under mild conditions, solutions to a VI correspond to Nash equilibria of a game. Similarly, in routing games, the Wardrop equilibria are solutions to the VI formed using the costs of each path.
In real-life, utilities or costs of players involved in a game may be uncertain and decisions must be made under this uncertainty. The behavior of the players may then depend on their risk-preferences and the involved cost functions are then risk measures of uncertain costs. Equilibrium in such scenarios corresponds to the solution of a VI, where each component of the map defining the VI is the risk associated to an appropriately defined function. Motivated by this setup, we consider VIs defined by the conditional value-at-risk (CVaR) of uncertain costs and develop stochastic approximation (SA) schemes to solve~them. 

\subsubsection*{Literature review}
General risk-based VIs, including CVaR-based VIs, and their potential applications are discussed in \cite{UR:14}. SA schemes are quite popular for solving \textit{stochastic variational inequality} (SVI) problems, see e.g.~\cite{UVS:13,YC-GL-YO:14} and references therein. In an SVI, the map associated to the VI is usually the expectation of an uncertain function and hence, an unbiased estimator of the map is available using a single sample of the uncertainty. As a result, the associated SA scheme enjoys strong convergence guarantees under fairly mild assumptions. However, this property does not hold in general for CVaR-based VIs. Instead, in our work, we use a finite number of samples to determine an empirical estimate of the CVaR. Such an estimator is biased but consistent and we use results from~\cite{YW-FG:10, RPR-NDS:10} to bound the deviation of the estimator from the true value of the CVaR. Recent works~\cite{RKK-LAP-SPB-KJ:19, LAP-KJ-RKK:19} generalize such bounds for more general distributions.

Closely related to our work,~\cite{AC:19-cdc} provides a sample average approximation (SAA) method for computing the solutions of the CVaR-based VI. In the SAA method, the CVaR is replaced with its empirical estimate and the solution of the VI formed using these empirical estimates is used to approximate the solution of the original problem. 
Our approach to analyze the convergence of the SA schemes proposed here involves approximating the asymptotic behaviour of a scheme by a trajectory of a continuous-time dynamical system and inferring convergence from the stability properties of the dynamical system~\cite{HJK-DSC:78},~\cite{VB:08}. 
In other related works,~\cite{AT-YC-MG-SM:17} and~\cite{CJ-LAP-FM-SM-CS:18} provide sample-based schemes for optimizing the CVaR and other general risk measures, respectively.

\subsubsection*{Statement of Contributions}
We start by defining the CVaR-based variational inequality (VI) where the map defining the VI consists of components that are the CVaR of uncertain functions. Our first contribution is the design of a ``vanilla'' stochastic approximation algorithm that, under strict monotonicity, asymptotically converges to a solution of the VI problem. The algorithm employs a sample-based estimator of the CVaR at each iteration. For convergence, the algorithm requires unbounded growth of the sample size as the iterations proceed.
To handle this limitation, our second and third contributions are the design of two stochastic approximation schemes, termed penalty-driven and multiplier-driven algorithms, that use the same estimator of the CVaR but use a bounded number of samples at each iteration. Under strict monotonicity, these algorithms are shown to converge asymptotically to a neighborhood of the solution set, the size of which can be tuned. The penalty-driven algorithm allows iterates to venture outside the set defining the VI but controls the deviation using a penalty term. On the other hand, the multiplier-driven algorithm, akin to primal-dual methods in optimization, ensures convergence of iterates to the set using multiplier variables. Our final contribution investigates the dependence of the size of the neighborhood that the iterates converge to on the empirical estimation error, the sample-size at each iteration, and the strong monotonicity parameter. 
A simulation example illustrates our result. 

\subsubsection*{Organization}
The paper is organized as follows. Section~\ref{section preliminaries} presents the notation and basic concepts on variational inequalities and the conditional value-at-risk. Section~\ref{section problem statement} describes the problem setup and provides a motivating example. Section~\ref{section stochastic approximation algorithms} proposes three stochastic approximation algorithms for solving the CVaR-based VI and provides, for two algorithms, the relationship between the estimation error, the sample size, and the accuracy. Section~\ref{section simulations} presents a simulation example. Finally, Section~\ref{section conclusion} describes our conclusions and ideas for future work.
\section{Preliminaries} \label{section preliminaries}
Throughout this paper we use the following notation. Let $\mathbb{R}$ and $\mathbb{N}$ denote the real and natural numbers, respectively. For $N \in \naturalnumbers$, we let $\until{N} := \{1, 2, \dots, N\}$. For given $x \in \mathbb{R}$, we use the notation $[x]_+ := \max(x,0)$. For $x \in \real^n$, we let $x_i$ denote the $i$-th element of $x$, and the $i$-th element of the vector $[x]_+$ is $[x_i]_+$. The Euclidean $2$-norm of $x$ is given by $\norm{x}$. For scalars $x, y \in \real$, the operator $[x]_{y}^+$ equals $x$ if $y > 0$ and it equals $\max\{0,x\}$ if $y = 0$. 
For vectors $x,y \in \real^n$, $[x]_{y}^+$ denotes the vector whose $i$-th element is $[x_i]_{y_i}^+$. The open $\epsilon$ neigborhood of $x$ is defined as ${\NN_{\epsilon}(x) = \setdef{y \in \real^n}{\norm{y-x} < \epsilon}}$. The Euclidean projection of $x$ onto the set $\HH$ is denoted $\Pi_\HH(x):=\argmin_{y \in \HH} \norm{x-y}$. The closure of a set $\SS \subset \real^n$ is denoted by $\cl(\SS)$.
\subsection{Variational inequalities, monotonicity, and KKT points}\label{sec:vi-kkt}
For a given map $F: \mathbb{R}^n \rightarrow \mathbb{R}^n$ and a closed set $\HH \subseteq \mathbb{R}^n$, the associated \textit{variational inequality} (VI) problem, $\VI(F,\HH)$, is to find ${h^* \in \HH}$ solving
\begin{equation*}
    (h - h^*)^\top F(h^*) \geq 0 \text{ for all } h \in \HH.
\end{equation*}
The set of all points that solve $\VI(F,\HH)$ is denoted $\SOL(F,\HH)$. An important concept in the context of VI's is \textit{monotonicity} of the map $F$. The map $F$
is called \textit{monotone}~if
\begin{equation*} \label{strict monotonicity}
    \big(F(x)  -  F(y)\big)^\top  (x  -  y)  \geq  0
\end{equation*}
holds for all ${x,y \in \mathbb{R}^n}$. If the inequality is strict for $x \not = y$, then $F$ is \textit{strictly monotone}. Similarly, we call $F$ \textit{strongly monotone} with constant ${c_F  >  0}$ if 
\begin{equation*}
    \big(F(x) - F(y)\big)^\top  (x-y)  \geq  c_F\norm{x - y}^2 
\end{equation*}
holds for all ${x, y \in \mathbb{R}^n}$. 
If $\HH$ is nonempty, compact, and $F$ 
is continuous, then $\SOL(F,\HH)$ is nonempty.
If $F$ is strictly monotone, then $\VI(F,\HH)$ has at most one solution \cite[Theorems 2.1 \& 2.2]{AN-DZ:96}.
Under the \textit{linear independence constraint qualification} (LICQ), we next characterize $\SOL(F,\HH)$ as the set of Karush-Kuhn-Tucker (KKT) points of $\VI(F,\HH)$.
\begin{lemma}\longthmtitle{KKT points of $\VI(F,\HH)$}\label{le:KKT}
Let 
    \begin{equation} \label{explicit contrained set 2}
    \HH := \setdef{h \! \in \! \real^n \!}{q^j(h) \! \leq \! 0, \! \enskip l^k(h) \! = \! 0, \enskip \forall j \! \in \! [s], \! \enskip k \! \in \! [t]},
    \end{equation}
    where the functions $q^j$, $l^k: \mathbb{R}^n \rightarrow \mathbb{R}$, $j \in \until{s}$, $k \in \until{t}$, are convex and affine, respectively, and continuously differentiable. 
    For $q(h) \! := \! (q_1(h), \dots, q_s(h))^\top \! \in \! \real^s$, let $Dq(h) \! \in \! \real^{s \times n}$ be its Jacobian at $h$, and similarly $Dl(h)$. For any ${h^* \in \real^n}$, if there exists a multiplier $(\lambda^*, \mu^*) \in \real^s \times \real^t$ satisfying 
    \begin{equation} \label{KKT}
        \begin{split}
             &F(h^*) + (Dq(h^*))^\top \lambda^* + (Dl(h^*))^\top \mu^* = 0, \\
             l(h^*&) = 0, \quad q(h^*) \leq 0, \quad  \lambda^{*} \geq 0, \quad \lambda^{*\top}q(h^*) = 0,
        \end{split}
    \end{equation}
    then we have $h^* \in \SOL(F,\HH)$. Such a point $(h^*,\lambda^*, \mu^*)$ is referred to as a KKT point of the $\VI(F,\HH)$. Conversely, for $h^* \in \SOL(F,\HH)$, let $\II_{h^*} = \setdef{j}{q^j(h^*) = 0}$.
    If the vectors $\{\nabla q^j(h^*)\}_{j \in \II_{h^*}}$ and $\{\nabla l^k(h^*)\}_{k \in [t]}$ are linearly independent, or in other words, the LICQ holds at $h^*$, then there exists a multiplier $(\lambda^*, \mu^*)$ satisfying \eqref{KKT}.
\end{lemma}
The above result is well-known in the context of convex optimization. The extension to the VI setting is forthright: the proof that the existence of a triplet $(\lambda^*,\mu^*,h^*)$ satisfying \eqref{KKT} is sufficient to guarantee $h^* \in \SOL(F,\HH)$ can be found in $\text{\cite[Proposition 3.46]{DS:18}}$. That the same condition is also necessary can be deduced from the result in the context of convex optimization (e.g. \cite[Theorem 12.1]{UF-WK-GS:10}) and noting that if $h^* \in \SOL(F,\HH)$, then it is also a minimizer of the function $y \mapsto y^\top F(h^*)$ subject to $y \in \HH$. 
\subsection{Conditional Value-at-Risk}
The \textit{Conditional Value-at-Risk} (CVaR) \textit{at level} $\alpha \in (0,1]$ of a real-valued random variable $Z$, defined on a probability space $(\Omega, \FF, \mathbb{P})$, is given by
\begin{equation*} 
\CVaR_\alpha[Z] := \inf_{t \in \mathbb{R}}\big\{t + \alpha^{-1} \mathbb{E}[Z - t]_+\big\},
\end{equation*}
where the expectation is with respect to $\Pb$.  The value $\alpha$ is a constant that characterizes risk-averseness, with smaller values of $\alpha$ giving a more risk-averse measure on $Z$. 
Given $N$ independent and identically distributed (i.i.d) samples $\{\widehat{Z}_j\}_{j \in [N]}$ of the random variable $Z$, one~can approximate $\CVaR_\alpha[Z]$ using the following empirical estimate
\begin{equation} \label{estimator CVaR}
    \CVaRhat^N_\alpha[Z] = \inf_{t \in \mathbb{R}}\big\{ t + (N \alpha)^{-1}  \textstyle\sum_{j=1}^N[\widehat{Z}_j - t]_+ \big\}.
\end{equation}
This estimator is biased, but consistent \cite[Chapter 6]{AS-DD-AR:14}. That is, the expected value of $\CVaRhat^N_\alpha[Z]$ is not necessarily equal to $\CVaR_\alpha[Z]$ and $\lim_{N \to \infty} \CVaRhat^N_\alpha[Z] = \CVaR_\alpha[Z]$ with probability one.

\section{Problem statement and motivating example} \label{section problem statement}
Consider a set of functions $C_i: \mathbb{R}^n \times \mathbb{R}^m \rightarrow \mathbb{R}$, ${i \in \until{n}}$, $(h,u) \mapsto C_i(h,u)$, where $u$ represents a random variable with distribution $\Pb$. For a fixed $h$, $C_i(h,u)$ is therefore a real-valued random variable. Define the map $\map{F_i}{\real^n}{\real}$ as the CVaR of $C_i$ at level $\alpha \in (0,1]$,
\begin{equation} \label{definition elements F}
       F_i(h) := \CVaR_\alpha\big[C_i(h,u)\big], \text{ for all } i \in \until{n}.
\end{equation}
 For notational convenience, let $C: \mathbb{R}^n \times \mathbb{R}^m \rightarrow \mathbb{R}^n$ and $\map{F}{\real^n}{\real^n}$ be the element-wise concatenation of the maps $\{C_i\}_{i \in \until{n}}$ and $\{F_i\}_{i \in \until{n}}$, respectively. Let $\HH \subseteq \real^n$ be a nonempty closed set.
The objective of this paper is to provide stochastic approximation (SA) algorithms to solve the variational inequality problem $\VI(F,\HH)$. Our strategy is to use an empirical estimator, derived from samples of $C(h,u)$, of the map $F$ at each iteration of the algorithm. Before we introduce the schemes, we will discuss a motivating example.

\subsubsection*{CVaR-based routing games~\cite{AC:19-cdc}} \label{section routing game}
Consider a directed graph $\GG = (\VV,\EE)$, where ${\VV = \until{\bar{n}}}$ is the set of vertices, and $\EE \subseteq \VV \times \VV$ is the set edges. To such a graph we associate a set $\WW \subseteq \VV \times \VV$ of origin-destination (OD) pairs. An OD-pair $w$ is given by an ordered pair $(v^w_o,v^w_d)$, where $v^w_o, v^w_d \in \VV$ are called the origin and the destination of $w$, respectively. The set of all paths in $\GG$ from the origin to the destination of $w$ is denoted $\PP_w$. The set of all paths
is given by $\PP = \cup_{w \in \WW} \PP_w$. Each of the participants, or agents, of the routing game is associated to an OD-pair, and can choose a path to travel from its origin to its destination. The choices of all agents give rise to a flow vector $h \in \mathbb{R^{|\PP|}}$. A common assumption in this context, which we will adopt here as well, is that the flow is non-atomic, meaning that each traffic participant controls an infinitesimal part of the flow. As a consequence, the flow $h$ is a continuous variable.

For each (OD)-pair $w$, a real value $d_w \geq 0$ defines the amount of traffic, or demand, associated to it.
The feasible set $\HH \subset \real^{|\PP|}$, containing all possible flows is then given by 
\begin{equation*}
\HH = \setdefb{h}{\sum_{p \in \PP_w}h_p = d_w, \enskip \forall w \in \WW, \enskip h_p \geq 0 \enskip \forall p \in \PP}.
\end{equation*}
To each of the paths $p \in \PP$, we associate a cost function $C_p: \mathbb{R}^{|\PP|} \times \mathbb{R}^{m} \rightarrow \mathbb{R}, (h,u) \mapsto C_p(h,u)$, which depends on the flow $h$, as well as on the uncertainty $u \in \mathbb{R}^m$.
Each agent chooses a path $p \in \PP_w$ that minimizes $\CVaR_\alpha\big[C_p(h,u)\big]$. These elements define the CVaR-based routing game to which we assign the following notion of equilibrium:
    the flow $h^* \in \HH$ is said to be a CVaR-based Wardrop equilibrium (CWE) of the CVaR-based routing game if the following hold for all $w \in \WW$:
    \begin{enumerate}
        \item $\sum_{p \in \PP_w} h^*_p = d_w$,
        \item $h^*_p > 0$ for $p \in \PP_w$ only if
        \begin{equation*}
            \CVaR_\alpha\big[C_p(h^*,u)\big] \leq \CVaR_\alpha\big[C_q(h^*,u)] \enskip \forall q \in \PP_w.
        \end{equation*}
    \end{enumerate}
The intuition behind this definition is that at equilibrium, for each agent, there is no path for which the CVaR of the cost is less than the CVaR of the cost on the selected path. Thus there is no incentive for the traffic participants to change their route choices. 
Under continuity of $C_p$, the set of CWE is equal to the set of solutions of $\VI(F,\HH)$, where $F: \mathbb{R}^{|\PP|} \rightarrow \mathbb{R}^{|\PP|}$ takes the form of~\eqref{definition elements F}.

\section{Stochastic approximation algorithms for solving $\VI(F,\HH)$} \label{section results} \label{section stochastic approximation algorithms}
In this section, we introduce the SA algorithms along with their convergence analysis. Then we establish results relating the accuracy of the algorithms to the size of the estimation error and the sample size in each iteration. All introduced schemes approximate $F$ with the estimator given in~\eqref{estimator CVaR}. Given $N$ independently and identically distributed samples $\big\{(\widehat{C_i(h,u)} )_j\big\}_{j=1}^N$ of the random variable $C_i(h,u)$, let 
\begin{align*}\label{eq:Fhat}
    \Fhat^N_i(h) := \inf_{t \in \mathbb{R}}\Big\{ t + \frac{1}{N\alpha}\sum_{j=1}^N\big[(\widehat{C_i(h,u)})_j - t\big]_+ \Big\}
\end{align*}
stand for the estimator of $F_i(h)$. We occasionally use $\CVaRhat_\alpha^N \big[C_i(h,u)\big]$ to denote $\Fhat^N_i(h)$.
Analogously, the estimator of $F(h)$ formed using the element-wise concatenation of $\Fhat^N_i(h)$, $i \in \until{n}$, is denoted by $\Fhat^N(h)$ or $\CVaRhat^N_\alpha \big[C(h,u)\big]$. We assume that the $N$ samples of each cost function are a result of the same set of $N$ events, that is, the distribution of $\Fhat^N(h)$ depends on $\Pb^N$.
 All algorithms introduced in this section depend on a sequence of step-sizes $\{\gamma^k\}_k^\infty$, where $\gamma^k > 0$ for all $k \in \naturalnumbers$. Common assumptions for this sequence are
 \begin{equation} \label{stepsize}
    \begin{split}
       \sum_{k = 0}^\infty \gamma^k = \infty, &\qquad \sum_{k = 0}^\infty (\gamma^{k})^2 < \infty.
    \end{split}
\end{equation}
For all upcoming algorithms, we will assume that the sequence $\{\gamma^k\}_k^\infty$ satisfies these assumptions.
 
 For a given sequence $\{N_k\}_{k=0}^\infty \subset \naturalnumbers$, and an initial vector $h^0 \in \HH$, the first algorithm under consideration, which we will refer to as the \textit{projected algorithm} is given by
\begin{equation}\label{eq:projected-original-form}
    h^{k+1} = \Pi_{\HH}\big( h^k - \gamma^k\Fhat^{N_k}(h^{k})\big),
\end{equation}
where $\Pi_{\HH}$ is the projection operator (cf. Section \ref{section introduciton}) and $h^k$ is the $k$-th iterate of $h$ produced by the algorithm. The above algorithm is inspired by the SA schemes for solving a stochastic VI problem, see~\cite{UVS:13} for details on other SA schemes. The key difference from the setup in~\cite{UVS:13} is the fact that there the map $F$ is the expected value of a random variable for which an unbiased estimator $\widehat{F}$ is available. In our case, the estimator is biased posing limitations on the sample requirements for convergence of the algorithms. 

For analysis, it is convenient to write the projected algorithm~\eqref{eq:projected-original-form} equivalently as
\begin{equation} \label{algorithm with error term}
    h^{k+1} = \Pi_{\HH}\Big( h^k - \gamma^k\big(F(h^{k}) + \widehat{\beta}^{N_k} \big) \Big),
\end{equation}
where $\widehat{\beta}^{N_k}$ is used to denote the error introduced by estimation, and is given by
\begin{equation}\label{error term in algorithm}
    \widehat{\beta}^{N_k} := \widehat{F}^{N_k}(h^k) - F(h^k).
\end{equation}

The following result on the convergence of \eqref{algorithm with error term} is then a direct consequence of \cite[Theorem 5.2.1]{HJK-GGY:97}.
\begin{proposition} \longthmtitle{Convergence of~\eqref{algorithm with error term} to a solution of $\VI(F,\HH)$} \label{original theorem}
Let $F$, as defined in \eqref{definition elements F}, be a strictly monotone, continuous function, and let $\HH$ be a compact convex set. 
For the algorithm~\eqref{algorithm with error term}, assume that the sequence of step-sizes $\{\gamma^k\}$ satisfies~\eqref{stepsize} and the sequence $\{N_k\}$ is such that $\{\widehat{\beta}^{N_k}\}$ is bounded with probability one and $N_k \to \infty$ as $k \to \infty$. Then, the iterates $\{h^k\}$ generated by~\eqref{algorithm with error term} satisfy
\begin{equation*}
    \lim_{k \rightarrow \infty}\norm{h^k - h^*} = 0,
\end{equation*}
 for $h^* \in \SOL(F,\HH)$, with probability one.
\end{proposition}
The proof proceeds in two steps: (a) showing that the iterative scheme asymptotically approaches a trajectory of a continuous-time dynamical system; and (b) establishing asymptotic stability of the system. We will use this reasoning for analyzing the other algorithms presented in this work.

Note that since the CVaR estimator is biased, Proposition~\ref{original theorem} requires the number of samples to grow unboundedly as the algorithm progresses. In order to address this tractability issue, we propose two algorithms which achieve convergence to an $\epsilon$ neighborhood of $\SOL(F,\HH)$ using a finite number of samples in each iteration. Unlike~\eqref{algorithm with error term}, both algorithms allow the iterates to take values outside the set $\HH$. This limitation is a result of our analysis approach. With projections and a biased estimator, the analysis of~\eqref{algorithm with error term} with finite samples at each iteration would involve studying the input-to-state stability of a continuous-time projected dynamical system~\cite{AN-DZ:96}, the theory for which is not yet available in the literature. 

 As in Proposition~\ref{original theorem}, we will impose continuity and monotonicity assumptions on $F$ in the upcoming results. Sufficient conditions for the Lipschitz continuity of $F$ are given in \cite[Lemma IV.8]{AC:19-cdc}.
 We provide the following general result on the continuity and monotonicity properties of $F$. 
\begin{lemma}\longthmtitle{Sufficient conditions for monotonicity and continuity of $F$}
    The following hold:
    \begin{itemize}
        \item Assume that for all $i \in [n]$, there exist functions ${f_i:\real^n \rightarrow \real}$ and $g_i: \real^m \rightarrow \real$ such that ${C_i(h,u) \equiv f_i(h) + g_i(u)}$.
        Let $f(h):= (f_1(h), \dots, f_n(h))$. Then, $F$ is monotone (resp. strictly or strongly monotone) if $f$ is monotone (resp. strictly or strongly monotone).
        \item If for any $\epsilon > 0$ there exist a $\delta > 0$ such that ${\norm{h - h'} \le \delta}$ implies $\norm{ C_i(h,u) - C_i(h',u)} \le \epsilon$ for all $i \in \until{n}$ and all $u$, then $F$ is continuous.
    \end{itemize}
\end{lemma}
The proof of the first part of this statement makes use of the coherence of the CVaR as a risk measure \cite[Page 261]{AS-DD-AR:14}. This, together with the made assumptions, implies
\begin{equation*}
    \CVaR_\alpha\big[C_i(h,u)\big] = f_i(h) + \CVaR_\alpha\big[g_i(u)\big].
\end{equation*}
As a consequence, we have $F(h) - F(h^*) = f(h) - f(h^*)$, and the result follows.
The part of this statement pertaining to continuity can be derived by arguments similar to those of the proof of \cite[Lemma IV.8]{AC:19-cdc}. Note that the given conditions for continuity may be difficult to check in practice. However, they hold when $\HH$ is compact, $u$ has a compact support, and for any fixed $u$, the functions $C_i$ are continuous with respect to $h$.
We now proceed to introduce two more algorithms. 

\subsection{Penalty-driven algorithm} 
We define the
\emph{penalty-driven algorithm} as
\begin{equation} \label{penalty algorithm}
    h^{k+1} = h^k -\gamma^k\Big(F(h^k) + c\big(h^k - \Pi_{\HH}(h^k)\big) + \widehat{\beta}^{N_k}\Big).
\end{equation}
Here $c > 0$ is a constant and the error sequence $\{\widehat{\beta}^{N_k}\}$ is as defined in~\eqref{error term in algorithm}. This algorithm allows the iterates $\{h^k\}$ to take values outside of $\HH$. However, the term $h^k - \Pi_{\HH}(h^k)$ controls this drift; the higher the value of the design parameter $c$, the closer the limit of $\{h^k\}$ is to $\HH$. In this sense, the constant $c$ determines the \emph{penalty} for moving out of the set.
Now we state the convergence properties of~\eqref{penalty algorithm}.
\begin{proposition} \label{penalty theorem} \longthmtitle{Convergence of the penalty-driven algorithm~\eqref{penalty algorithm}}
Let $F$, as defined in \eqref{definition elements F}, be a strictly monotone, continuous function, and let $\HH$ be a compact convex set. For the algorithm \eqref{penalty algorithm} and $h^* \in \SOL(F,\HH)$, let $c = d\norm{F(h^*)}$ for some $d > 0$. Assume that the sequence of step-sizes $\{\gamma^k\}$ satisfies \eqref{stepsize} and that the sequence $\{N^k\}$ is such that $\{\widehat{\beta}^{N_k}\}$ and $\{h^k\}$ are bounded with probability one. 
Then, for any $\epsilon > 0$, there exist $d_{\epsilon} > 0$ and $N_{\epsilon} \in \mathbb{N}$ such that $d \geq d_{\epsilon}$ and $N_k \geq N_{\epsilon}$ for all $k$ imply, with probability one, 
\begin{equation} \label{penalty epsilon statement}
    \lim_{k \rightarrow \infty} \norm{h^k - h^*} \leq \epsilon.
\end{equation}
\end{proposition}
\begin{proof}
For convenience, we split the error as ${\widehat{\beta}^{N_k} = b^{N_k} + \widehat{\xi}^{N_k}}$, where ${b^{N_k} = \mathbb{E}[ \widehat{\beta}^{N_k}]}$. By definition, $\mathbb{E}[\widehat{\xi}^{N_k}] = 0$ and by the boundedness assumption, there exists a constant $B > 0$ such that $\norm{b^{N_k}} \le B$ for all $k$. As mentioned before, the proof proceeds in two steps. First, the sequence $\{h^k\}$ is shown to converge to a trajectory of the following continuous-time system
\begin{equation} \label{ODE penalty} 
          \dot{\bar{h}} (t) = -F\big(\bar{h}(t)\big) - c\Big(\bar{h}(t) - \Pi_{\HH}\big(\bar{h}(t)\big)\Big) - \bar{b}(t). 
 \end{equation}
Here, $\bar{b}(\cdot)$ is a uniformly bounded map satisfying $\norm{\bar{b}(t)} \leq B$ for all $t$.
The proof of the above fact is analogous to that of \cite[Theorem 5.3.1]{HJK-DSC:78} and we avoid repeating these arguments for space reasons. To be more precise, the sequence $\{h^k\}$ converges to a trajectory of~\eqref{ODE penalty} if there exists a map $t \mapsto \bar{h}(t)$ with $\bar{h}(0) \in \real^n$ satisfying~\eqref{ODE penalty} for all $t$ and 
\begin{equation} \label{convergence to trajectory}
       \textstyle \lim_{k \rightarrow \infty} \sup_{j \geq k}\norm{h^j - \bar{h}(\sum_{r=k}^{j-1} \gamma^r)} = 0.
\end{equation}
That is, the discrete-time trajectory formed by the linear interpolation of the iterates $\{h^k\}$ approaches a continuous-time trajectory $t \mapsto \bar{h}(t)$. Convergence of the sequence $\{h^k\}$ can then be analyzed by studying the asymptotic stability of~\eqref{ODE penalty}. 
To this end, define the Lyapunov function
    \begin{equation} \label{lyapunov function penalty}
        V\big(\bar{h}\big) = \norm{\bar{h} - h^*}^2,
    \end{equation}
 where $h^*$ is the unique solution to $\VI(F,\HH)$, that follows from strict monotonicity. For readability, we introduce the notation ${\tilde{h} = \Pi_{\HH}(\bar{h})}$. We first analyze the evolution of $V$ along~\eqref{ODE penalty} when $\bar{b} \equiv 0$. For notational convenience, define the right-hand side of~\eqref{ODE penalty} in such a case by the map ${\map{X_{b\equiv0}}{\real^n}{\real^n}}$. The Lie derivative of $V$ along $X_{b\equiv0}$ is
   \begin{equation} \label{vdot penalty}
        \nabla V(\bar{h})^\top X_{b\equiv0}(\bar{h}) = -2(\bar{h} - h^*)^\top\big(F(\bar{h}) + c(\bar{h} - \tilde{h})\big).
    \end{equation}
   Our next step is to show that the Lie derivative is upper bounded by a negative quantity whenever the trajectory is at least $\epsilon > 0$ away from $h^*$. To this end, let ${\Delta_\epsilon :=\setdef{\bar{h}}{\bar{h} \notin \NN_{\epsilon}(h^*), \enskip \norm{\bar{h}} \leq B_h}}$, where $\NN_{\epsilon}(h^*)$ is an (open) $\epsilon$ neighborhood of $h^*$ and $B_h$ is a bound on the trajectory $\bar{h} (\cdot)$ where the iterates converge to. Such a bound exists due to \cite[Theorem 5.3.1]{HJK-DSC:78}. Note that $\Delta_\epsilon$ is a compact set.
   We will show that there exists $\delta_1> 0$ such that ${\nabla V(\bar{h})^\top X_{b\equiv0}(\bar{h})  < - \delta_1}$ for all $\bar{h} \in \Delta_\epsilon$. Note that we have $(\bar{h} - h^*)^\top F(\bar{h}) > (\bar{h} - h^*)^\top F(h^*)$ due to strict monotonicity of $F$. Using this fact in~\eqref{vdot penalty}, gives for all $\bar{h} \in \Delta_\epsilon$,
   \begin{align*}
        \nabla V(\bar{h})&^\top X_{b\equiv0}(\bar{h}) < -2(\bar{h} - h^*)^\top\big(F(h^*) + c(\bar{h} - \tilde{h})\big) \\
        &= -2(\bar{h} - \tilde{h})^\top F(h^*) -2(\tilde{h} - h^*)^\top F(h^*) \nonumber \\
        &\qquad -2c(\bar{h} - \tilde{h})^\top (\bar{h} - \tilde{h}) -2c(\tilde{h} - h^*)^\top(\bar{h} - \tilde{h}),
    \end{align*}
   where we recall that $\tilde{h} = \Pi_{\HH}(\bar{h})$. Since $h^* \in \SOL(F,\HH)$, we have $(\tilde{h} - h^*)^\top F(h^*) \geq 0$. Further, since $\HH$ is convex, the projection property implies that $(\tilde{h} - h^*)^\top (\bar{h} - \tilde{h}) \ge 0$. Using these two facts and applying the Cauchy-Schwartz inequality in the above derived inequality, we obtain, for all $\bar{h} \in \Delta_\epsilon$,
    \begin{equation*}
       \nabla V(\bar{h})^\top \! X_{b\equiv0}(\bar{h}) \! < \! 2\norm{\bar{h} \! - \! \tilde{h}}\big(\norm{F(h^*)} -d\norm{F(h^*)}\norm{\bar{h} - \tilde{h}}\big).
   \end{equation*}
   Since $\Delta_\epsilon$ is compact, there exists $\delta_0 > 0$ such that
   \begin{align*}
       \nabla V(\bar{h})^\top & X_{b\equiv0}(\bar{h}) \leq  \\  
       &2\norm{\bar{h} - \tilde{h}} \norm{F(h^*)} -2d\norm{F(h^*)}\norm{\bar{h} - \tilde{h}}^2 - \delta_0 
   \end{align*}
    holds for all $\bar{h} \in \Delta_\epsilon$. The right-hand side as a function of $\norm{\bar{h} - \tilde{h}}$ attains a maximum at ${\norm{\bar{h}- \tilde{h}} = \frac{1}{2d}}$. Thus, we have 
    
\begin{equation*}
    \nabla V(\bar{h})^\top X_{b\equiv0}(\bar{h}) \leq \frac{1}{2d}\norm{F(h^*)} -\delta_0.
\end{equation*}
    It follows that if we set
    \begin{equation*}
        d = \frac{\norm{F(h^*)}}{2 \delta_0}(1 - \frac{\delta_1}{\delta_0})^{-1}
    \end{equation*}
    for some $\delta_1 < \delta_0$, then $\nabla V(\bar{h})^\top X_{b\equiv0}(\bar{h}) \leq -\delta_1$ for all ${\bar{h} \in \Delta_\epsilon}$. At this point, we drop the earlier made restriction $b \equiv 0$. Denote the right-hand side of~\eqref{ODE penalty} by the map $X$. From the above reasoning, the Lie derivative of $V$ along~\eqref{ODE penalty} satisfies 
    \begin{align}\label{eq:lie-with-b}
        \nabla V(\bar{h})^\top X(\bar{h}) \le - \delta_1 + 2 (\bar{h} - h^*)^\top b
    \end{align}
    for all $\bar{h} \in \Delta_\epsilon$.
Note that as stated before, the bound on the norm of the map $b$ is the same as the bound on the iterates $\{b^{N_k}\}$. Since $\bar{h}(\cdot)$ is bounded and for any $k$, $\norm{b^{N_k}} \to 0$ as $N_k \to \infty$ (as the empirical estimate of the CVaR is consistent), we conclude that selecting large enough $N_k$ for all $k$, implies $\norm{2(\bar{h} - h^*)^\top b} < \delta_1$. Plugging this inequality in~\eqref{eq:lie-with-b} yields, for all $\bar{h} \in \Delta_\epsilon$,
\begin{equation*}
    \nabla V(\bar{h})^\top X(\bar{h}) \le - \delta
\end{equation*}
 for some $\delta \in (0,\delta_1)$. This implies that the trajectory $\bar{h}(\cdot)$ of~\eqref{ODE penalty} reaches the closure of $\NN_\epsilon(h^*)$ in finite time. 
Additionally, once the trajectory reaches $\cl (\NN_{\epsilon}(h^*))$, it stays there. 
Thus,~\eqref{penalty epsilon statement} holds with probability one, concluding the proof.
\end{proof}
\begin{remark}\longthmtitle{Practical considerations of~\eqref{penalty algorithm}}\label{equivalent to bounded problem}
In Proposition \ref{penalty theorem}, for small values of $\epsilon$, one would require a large value of $d$ to ensure convergence. This may result in large oscillations of $h^k$ when the term $\gamma^k d$ remains large. Such behaviour can be prevented by either starting with small values of $\gamma^k$ or increasing $d$ along iterations, until it reaches a predetermined size. The result is then still valid but the convergence can only be guaranteed once $d$ reaches the required size. \oprocend 
\end{remark}
\begin{remark}\longthmtitle{Generalizations of Proposition \ref{penalty theorem}} \label{remark on projection}
For all our convergence results, we require $\{h^k\}$ to be bounded. This is however not very restrictive. The results remain valid if the sequence $\{h^k\}$ is projected onto a hyper-rectangle containing $\HH$ (cf. \cite[Page 40]{HJK-DSC:78}). 
\oprocend
\end{remark}
\subsection{Multiplier-driven algorithm}
Algorithms \eqref{algorithm with error term} and \eqref{penalty algorithm} involve projection onto the set $\HH$ at each iteration. This can be computationally burdensome.
Our next algorithm overcomes this limitation. Inspired by the Lagrangian method, we assume $\HH$ to be of the form \eqref{explicit contrained set 2} and introduce a multiplier variable $(\lambda, \mu) \in \realnonnegative^s \times \real^t$ that enforces constraint satisfaction as the algorithm progresses. In order to simplify the coming equations, we write
\begin{equation*}
H(h,\lambda,\mu) := F(h) + Dq(h)^\top \lambda + Dl(h)^\top \mu, 
\end{equation*} 

where $Dq(h)$ and $Dl(h)$ are the Jacobians of $q$ and $l$ at $h$, respectively. The \emph{multiplier-driven algorithm} is given by 
\begin{equation} \label{Lagrangian algorithm}
    \begin{split}
        h^{k + 1} &= h^k - \gamma^k \big(H(h^k,\lambda^k,\mu^k) + \widehat{\beta}^{N_k}\big),  \\
        \lambda^{k+1} &= \big[\lambda^{k} + \gamma^k q(h^k) \big]_+, 
        \\
        \mu^{k+1} & = \mu^k + \gamma^k l(h^k). 
    \end{split} 
\end{equation} 
Recall that $\widehat{\beta}^{N_k}$ is the error due to empirical estimation of $F$.
 The next result states the convergence properties of~\eqref{Lagrangian algorithm} to a KKT point of the VI (see Section~\ref{section preliminaries} for definitions).
 \begin{proposition} \longthmtitle{Convergence of the multiplier-driven algorithm~\eqref{Lagrangian algorithm}} \label{lagrange theorem}
    Let $F$, as defined in \eqref{definition elements F}, be a strictly monotone, continuous function, and let $\HH$ be a compact set given by \eqref{explicit contrained set 2}, where the functions $q^j$, $j \in \until{s}$, are affine. Assume that the LICQ holds for ${h^* \in \SOL(F,\HH)}$, and let $(h^*,\lambda^*, \mu^*)$ be an associated KKT point. 
    For algorithm \eqref{Lagrangian algorithm}, assume that the sequence of step-sizes $\{\gamma^k\}$ satisfies \eqref{stepsize} and let $\{N_k\}$ be such that $\{\widehat{\beta}^{N_k}\}$, $\{h^k\}$, $\{\lambda^k\}$ and $\{\mu^k\}$ are bounded with probability~one.
    Then, for any $\epsilon > 0$, there exists an $N_\epsilon \in \mathbb{N}$ such that if $N_k \geq N_\epsilon$ for all $k$, then, with probability one,
    \begin{equation} \label{lagrangian epsilon statement}
        \lim_{k \rightarrow \infty} \norm{h^k - h^*} \leq \epsilon. 
    \end{equation} 
 \end{proposition}
 \begin{proof}
  Analogous to the proof of Proposition~\ref{penalty theorem}, the first step establishes convergence with probability one of 
  the sequence $\big\{(h^k, \lambda^k, \mu^k)\big\}$, in the sense of \eqref{convergence to trajectory}, to a trajectory $\big(\bar{h}(\cdot),\bar{\lambda}(\cdot),\bar{\mu}(\cdot)\big)$ of the following dynamics 
  \begin{subequations}\label{eq:ODE-h-lm}
  \begin{align}
        \dot{\bar{h}}(t) &= -H\big(\bar{h}(t), \bar{\lambda}(t), \bar{\mu}(t)\big)  - \bar{b}(t)    \\
          \dot{\bar{\lambda}}(t) &= \Big[q\big(\bar{h}(t)\big)\Big]_{\bar{\lambda}(t)}^+, \quad \dot{\bar{\mu}}(t) = l\big(\bar{h}(t)\big), \label{ODE lambda part}
 \end{align}
 \end{subequations}
 where $\bar{\lambda}(\cdot)$ is contained in the nonnegative orthant due to the projection. 
 The map $\bar{b}(\cdot)$ is uniformly bounded. Specifically, since $\{\widehat{\beta}^{N_k}\}$ is bounded, there exists a $B > 0$ such that $\norm{b^{N_k}} \leq B$ for all $k$, where ${b^{N_k} = \mathbb{E}[ \widehat{\beta}^{N_k}]}$. We then have $\norm{\bar{b}(t)} \le B$ for all $t$. 
The proof of convergence of the iterates to a continuous trajectory is similar to that of \cite[Theorem 5.2.2]{HJK-DSC:78} and is not repeated here for brevity. Next, we analyze the convergence of~\eqref{eq:ODE-h-lm}. We will occasionally use $\bar{x}$ as shorthand for $(\bar{h},\bar{\lambda},\bar{\mu})$.
     Define the Lyapunov function
    \begin{equation} \label{lyapunov function}
        V(\bar{h},\bar{\lambda}, \bar{\mu}) := \norm{\bar{h} - h^*}^2 + \norm{\bar{\lambda} - \lambda^*}^2 + \norm{\bar{\mu} - \mu^*}^2,
    \end{equation}
    where $h^*$ is the unique solution of $\VI(F,\HH)$ and $(h^*,\lambda^*, \mu^*)$ is an associated KKT point.
 We analyze the evolution of \eqref{lyapunov function} for the case $\bar{b} \equiv 0$. Denoting the right-hand side of~\eqref{eq:ODE-h-lm} by the map $X_{b \equiv 0}$, we get the Lie derivative of $V$ along~\eqref{eq:ODE-h-lm} as
    \begin{equation} \label{lagrange proof vdot}
    \begin{split}
    & \nabla V(\bar{x})^\top  X_{b \equiv 0}(\bar{x}) = -2(\bar{h} - h^*)^\top H(\bar{h},\bar{\lambda}, \bar{\mu})  \\ 
     \! \!+ 2(&\bar{\lambda}  \! - \! \lambda^*)^\top \! \big(q(\bar{h}) \! + \! [q(\bar{h})]_{\bar{\lambda}}^+ \! - \! q(\bar{h})\big) \! + \! 2(\bar{\mu} \! - \! \mu^*)^\top \! l(\bar{h}). 
    \end{split}
    \end{equation}
    Note that for any $j$, $([q(\bar{h})]_{\bar{\lambda}}^+)_j = (q(\bar{h}))_j$ if $\bar{\lambda}_j > 0$. Further, if $\bar{\lambda}_j = 0$, then $\bar{\lambda}_j - \lambda^*_j \le 0$. Thus we see that $(\bar{\lambda} \! - \! \lambda^*)^\top([q(\bar{h})]_{\bar{\lambda}}^+ \! - \! q(\bar{h})) \leq 0$. Since $q$ and $l$ are affine, we have $Dq(\bar{h}) = Dq(h^*)$ and $Dl(\bar{h}) = Dl(h^*)$ for all $\bar{h} \in \real^n$. Combining this with strict monotonicity we get for $\bar{h} \not = h^*$
    \begin{equation*}
        \begin{split}
            0&<(\bar{h} - h^*)^\top \big(H(\bar{h},\bar{\lambda}, \bar{\mu})
            - H(h^*,\bar{\lambda}, \bar{\mu}) \big) \\
            &=(\bar{h} - h^*)^\top \big(H(\bar{h},\bar{\lambda}, \bar{\mu})
             - H(h^*,\lambda^*, \mu^*)\\ 
            &+ Dq(h^*)^\top\lambda^* - Dq(h^*)^\top\bar{\lambda} + Dl(h^*)^\top\mu^* - Dl(h^*)^\top\bar{\mu}\big). \\
        \end{split}
    \end{equation*}
    Using \eqref{KKT}, and the assumption that functions are affine gives
    \begin{equation*} 
    \begin{split}
                &- (\bar{h}  -  h^*)^\top H(\bar{h},\bar{\lambda}, \bar{\mu})    \\    
        &<  
        (\lambda^* - \bar{\lambda})^\top \big(q(\bar{h}) - q(h^*)\big) + (\mu^* - \bar{\mu})^\top \big(l(\bar{h}) - l(h^*) \big).
    \end{split}
    \end{equation*}
    Combining these derivations, we get that for $\bar{h} \not = h^*$,
    \begin{equation*}
        \nabla V(\bar{x})  X_{b \equiv 0}(\bar{x})  <  2(\bar{\lambda} -\lambda^{*} )^\top  q(h^*)  +  2(\bar{\mu} -\mu^{*})^\top  l(h^*).
    \end{equation*}
    From \eqref{KKT} we have $2\lambda^{*\top }q(h^*) = 0$, $2\bar{\lambda}^\top q(h^*) \leq 0$ and $l(h^*) = 0$, which implies $\nabla V(\bar{h},\bar{\lambda},\bar{\mu})  X_{b \equiv 0}(\bar{h},\bar{\lambda},\bar{\mu}) < 0$ whenever $\bar{h} \not = h^*$. The rest of the proof is analogous to the corresponding section of the proof of Proposition \ref{penalty theorem}.
 \end{proof}
 \begin{remark} \longthmtitle{Generalizations of Proposition~\ref{lagrange theorem}} \label{remark generalization}
    In Proposition \ref{lagrange theorem} we require boundedness of $\{h^k\}$, $\{\lambda^k\}$ and $\{\mu^k\}$. 
    Similar to the case in Remark \ref{remark on projection}, when upper bounds on $\norm{\lambda^*}$ and $\norm{\mu^*}$ are known beforehand,
    projection onto hyper-rectangles can be used to ensure boundedness, while the result remains valid, (cf.~\cite[Theorem 5.2.2]{HJK-DSC:78}).
    \oprocend
 \end{remark}
\subsection{Estimation error, sample sizes, and accuracy}
For both algorithms \eqref{penalty algorithm} and \eqref{Lagrangian algorithm}, the convergence depends on the bias of the used estimator, given by ${b_{N_k} := \mathbb{E}[\widehat{\beta}^{N_k}]}$. When $F$ is assumed to be strongly monotone, we can give an explicit bound on $\norm{b^{N_k}}$ sufficient for ensuring convergence.
\begin{corollary} \label{corollary bound on bias} \longthmtitle{Estimation error bounds under strong monotonicity}
    Assume that $F$ is strongly monotone with constant $c_F$. For given sequences $\{h^k\}$ and $\{N^k\}$, define 
    \begin{align}\label{eq:hmax}
        &h_+ := \max_{h \in \{h^k\},h' \in \HH}\norm{h - h'}, &b_{N_k} = \mathbb{E}[\widehat{\beta}^{N_k}].
    \end{align}
    The following then hold for any $\epsilon > 0$:
    \begin{enumerate}
        \item Assume the conditions of Proposition \ref{penalty theorem} and consider $\{h^k\}$ generated by \eqref{penalty algorithm}. Let $h_+$ and $b_{N_k}$ be given by~\eqref{eq:hmax} and let ${d = c_d\frac{\norm{F(h^*)}}{2}}$, with $c_d > 1$. Then
        \begin{equation*}
        \norm{b^{N_k}} < (1 - \frac{1}{c_d})\frac{c_F\epsilon^2}{h_+} \text{ for all } k \in \naturalnumbers
        \end{equation*}
        implies $\lim_{k \rightarrow \infty} \norm{h^k - h^*} \leq \epsilon$ with probability~one.
        \item Assume the conditions of Proposition \ref{lagrange theorem} and consider $\{h^k\}$ generated by \eqref{Lagrangian algorithm}. Let $h_+$ and $b_{N_k}$ be given by \eqref{eq:hmax}. Then
        \begin{equation*}
        \norm{b^{N_k}} < \frac{c_F\epsilon^2}{h_+} \text{ for all } k \in \naturalnumbers
        \end{equation*}
         implies $\lim_{k \rightarrow \infty} \norm{h^k - h^*} \leq \epsilon$ with probability~one. 
    \end{enumerate}
\end{corollary}
\begin{proof}
 For the first statement the proof is analogous to that of Proposition \ref{penalty theorem}. Due to strong monotonicity, we have
  \begin{equation*}
        -2(\bar{h} - h^*)^\top F(\bar{h}) \leq -2c_F\norm{\bar{h} - h^*}^2.
     \end{equation*}
    Setting $\delta_0 = 2c_F\epsilon^2$ in the proof of Proposition \ref{penalty theorem}, we get $\delta_1 = (1 - \frac{1}{c_d})2c_F\epsilon^2$. Since $\{h^k\}$ is bounded by assumption, we have $h_+ < \infty$. The result then follows. 
    Similar reasoning holds for \eqref{Lagrangian algorithm}.
\end{proof}

We would now like to translate the condition imposed in the above result on $\norm{b^{N_k}}$ into a condition on the sample requirement $N_{\epsilon}$. To this end, under compactness, we give a result supplying a bound on $\norm{b^{N_k}}$ depending on $N^k$. 
\begin{lemma} \longthmtitle{Relation between estimation error and sample size} \label{lemma bound}
Let $F$ be as defined in \eqref{definition elements F}, where ${C_i(h,u) \in [z_1, z_2]}$, $z_2 \ge z_1$, for all $h,u$ and $i$. Then,
for ${b^{N_k} = \mathbb{E}\big[F(h^k) - \widehat{F}^{N_k}(h^k)\big]}$, we have
\begin{equation*}
    \norm{b^{N_k}} \leq \frac{3}{2}\sqrt{\frac{5 n \pi}{N_k \alpha}}(z_2 - z_1).
\end{equation*}
\end{lemma}
\begin{proof}
We will proof the result for the case $n = 1$. The generalization to $n > 1$ is straightforward. For a scalar random variable $Z$, we have $\mathbb{E}\big[\widehat{\CVaR}_\alpha[Z]\big] \leq \CVaR_\alpha[Z]$ \cite[Equation 5.22]{AS-DD-AR:14}. Therefore, 
\begin{align*}
    \Big| \mathbb{E}\big[\CVaR_\alpha[Z] &- \widehat{\CVaR}_\alpha[Z]  \big] \Big| \\
    & = \mathbb{E}\big[\CVaR_\alpha[Z] - \widehat{\CVaR}_\alpha[Z]  \big]    \\     
    &\leq \mathbb{E}\big[\CVaR_\alpha[Z] - \widehat{\CVaR}_\alpha[Z] \big]_+.
\end{align*}
From \cite[Theorem 3.1]{YW-FG:10}, we have the concentration bound
\begin{equation} \label{equation concentration bound}
    \mathbb{P}\big[\CVaR_\alpha[Z] - \widehat{\CVaR}_\alpha[Z] \! \geq \! z \big] \leq 3 e^{-\frac{1}{5}\alpha \left(\frac{z}{z_2 - z_1}\right)^2 N}.
\end{equation}
Thus we have
\begin{align*}
    &\mathbb{E}\big[\CVaR_\alpha[Z] - \widehat{\CVaR}_\alpha[Z] \big]_+  \\
    &= \int_0^{\infty}\mathbb{P}\big[\CVaR_\alpha[Z] - \widehat{\CVaR}_\alpha[Z]\big] \geq z]dz \\
    &\leq \frac{3}{2}\sqrt{\frac{5 \pi}{N_k \alpha}}(z_2 - z_1).
\end{align*}
The last inequality can be obtained by calculating the integral of the right-hand side of \eqref{equation concentration bound} using polar coordinates. The details of this derivation are omitted. The result then follows.
\end{proof}
 The above result leads to a lower bound on $N_k$, for both algorithms \eqref{penalty algorithm} and \eqref{Lagrangian algorithm}, that ensures convergence to $\NN_\epsilon(h^*)$.
\begin{corollary} \longthmtitle{Sample size bounds under strong monotonicity}
Let $F$ be strongly monotone with constanct $c_F$, and assume the conditions of Proposition \ref{penalty theorem} (resp. Proposition \ref{lagrange theorem}) and Lemma \ref{lemma bound} hold. For $\{h^k\}$ generated by \eqref{penalty algorithm} (resp. \eqref{Lagrangian algorithm}) define $h_+$, $b_{N_k}$, and $c_d$ as in Corollary \ref{corollary bound on bias}. For $\epsilon > 0$
\begin{align*}
    N_k  &> \frac{45 n \pi}{4\alpha}  \left(\frac{h_+ (z_2 - z_1)}{\epsilon^2 c_F  (1 - \frac{1}{c_d})}\right)^2  \text{ for all } k \in \naturalnumbers,
    \\
    (\text{resp. } N_k  &>  \frac{45 n \pi}{4\alpha} \left(\frac{h_+ (z_2 - z_1)}{\epsilon^2 c_F }\right)^2 \text{ for all } k \in \naturalnumbers,)
\end{align*}
implies ${\lim_{k \rightarrow \infty}\norm{h^k - h^*} \leq \epsilon}$ with probability one.
\end{corollary}
\section{Simulations} \label{section simulations}
\begin{figure}     
    \centering
    \includegraphics[width = 0.9\columnwidth]{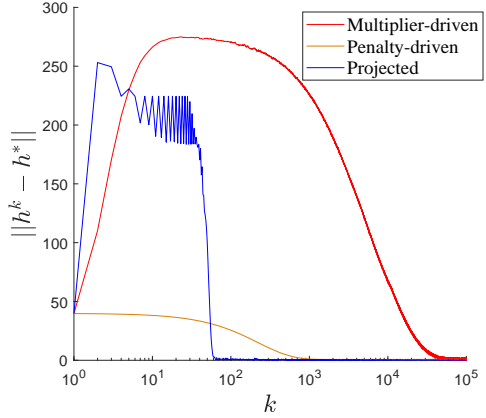}
    \vspace*{-1ex}
    \caption{Plot illustrating the convergence of the algorithms for the example given in Section~\ref{section simulations}.
    All algorithms performed 200000 iterations, and  
    $N_k = 100$. For the penalty-driven algorithm, we set ${c = 30000}$. We used $\gamma^k = \frac{1}{k}$, $\gamma^k = \frac{1000}{k + 10^7}$ and $\gamma^k = \frac{1000}{k + 2*10^5}$ for the projected, penalty-driven and multiplier driven algorithm, respectively.}
\label{figure logscale}
\end{figure}
\begin{figure}
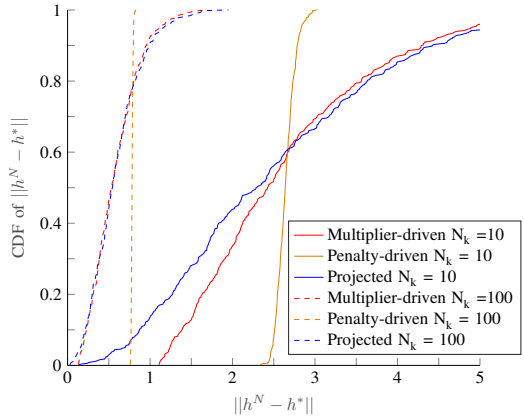
     
    \centering
    \vspace*{-36ex}
    \include{combined}
    \vspace*{-4ex}
    \caption{This plot illustrates the convergence of the discussed algorithms to $h^* \in \SOL(F,\HH)$ of an example $\VI$ (cf. Section \ref{section simulations}). The lines show the CDF attained for the different algorithms using 500 runs. On the horizontal axis we have the norm of the difference between the final iterate $h^N$ and $h^*$, and on the vertical axis we have the fraction of runs that achieved that precision. For the projected, penalty-driven and multiplier driven algorithm with $N_k = 10$, we used $N=200$, $N=1000$ and $N=200000$ respectively, and for $N_k = 100$ we set $N=1400$ for the penalty-driven algorithm. Different values are used since the algorithms require a different number of iterations to converge. For the penalty-driven algorithm we used $c = 10000$ and $c = 30000$ for the cases $N_k = 10$ and $N_k = 100$. For all runs the initial error is given by ${\norm{h^0- h^*} = 39.72}$. We note that the all curves increase to $1$ on the vertical axis. This is not shown here as we focus on a smaller $x$-axis range to better emphasize all curves.}
    \vspace*{-4ex}
\label{figure CDF_N500}
\end{figure}

Here we demonstrate the application of the stochastic approximation algorithms for finding the solutions of a CVaR-based variational inequality. The example is an instance of a $\CVaR$-based routing game as introduced in Section \ref{section routing game}. The example discussed is taken from \cite{AC:19-cdc}, where it was adapted from \cite[Section 6.3]{YX-UVS:15}. It consists of a simple network of two nodes $\VV = \{A,B\}$, and five edges. The edges $\{1,2,3\}$ go from $A$ to $B$, and edges $\{4,5\}$ go from $B$ to $A$. Then $\PP = \{1,2,3,4,5\}$. The demand equals 260 from $A$ to $B$, and 170 from $B$ to $A$, giving us the feasisble set
\begin{equation*}
    \HH = \setdef{h \in \realnonnegative^5}{h_1 + h_2 + h_3 = 260, \enskip h_4 + h_5 = 170}.
\end{equation*}
The cost functions are given by
\begin{equation*}
    C(h,u) = \left(\begin{array}{c}
    40h_1 + 20h_4 + 1000 + 3000u_1 \\
    60h_2 + 20h_5 + 950\\
    80h_3 + 3000\\
    8h_1 + 80h_4 + 1000 + 4000u_2\\
    4h_2 + 100h_5 + 1300
    \end{array}\right),
\end{equation*}
where the variables $u_1$ and $u_2$ model the uncertainty in the system, and are independent and uniformly distributed random variables on $[0,1]$. Setting $\alpha = 0.2$, we have defined the considered $\CVaR$-based routing game. From \cite{AC:19-cdc}, the CWE (see Section \ref{section routing game}) of this routing game is given by ${h^* = (89.52,98.39,72.09,74.32,95.68)}$.

Figure \ref{figure logscale} shows the evolution of the error of the considered algorithms for a single run. Different step-size sequence were used, in order to avoid unstable behaviour. The figure shows that all algorithms converge to a neighborhood of the solution of the VI, albeit with a different number of iterations.
Using the same step-size sequences as for Figure~\ref{figure logscale}, Figure~\ref{figure CDF_N500} shows the empirical cumulative distribution function (CDF) of the distance of the last iterate of the algorithm to the solution. We infer that as the number of samples used per iteration increases, the last iterates gets closer to the solution. Note that even though $N_k$ does not grow unboundedly, the projected algorithm still converges. 
\section{Conclusions} \label{section conclusion}
We have considered variational inequalities defined by the CVaR of cost functions and provided two stochastic approximation algorithms for solving them. We have analyzed the asymptotic convergence of these algorithms when, at each iteration, only finite number of samples are used to estimate the CVaR. We have carefully specified the trade-off between the sample requirement and the accuracy of the algorithms.

Future work will focus on analyzing the finite-time properties of the introduced algorithms. We wish to also explore input-to-state stability of projected dynamical systems. 
\bibliography{bibfile}
\bibliographystyle{ieeetr}
\end{document}